\documentclass[12pt, leqno]{amsart}
\usepackage{amsmath}
\usepackage{amssymb}
\usepackage{amsthm}
\usepackage{enumerate}
\usepackage[mathscr]{eucal}
\usepackage{xcolor}
\theoremstyle{plain}
\usepackage{tikz}
\usepackage{soul}
\usepackage[normalem]{ulem}
\usepackage[normalem]{ulem}
\newtheorem{theorem}{Theorem}[section]
\newtheorem{lemma}[theorem]{Lemma}
\newtheorem{prop}[theorem]{Proposition}
\newtheorem{cor}[theorem]{Corollary}

\theoremstyle{definition}
\newtheorem{definition}[theorem]{Definition}
\newtheorem{remark}[theorem]{Remark}

\newtheorem{example}[theorem]{Example}

\theoremstyle{remark}

\numberwithin{equation}{section}
\allowdisplaybreaks

\begin{document}
	\title[Smoothness of bounded linear operators on  semi-Hilbert space]{Smoothness in the space of bounded linear operators on  semi-Hilbert space}
		\author[Barik,  Ghosh, Paul and Sain]{Somdatta Barik, Souvik Ghosh, Kallol Paul and Debmalya Sain}
	
     \address[Barik]{Department of Mathematics\\ Jadavpur University\\ Kolkata 700032\\ West Bengal\\ INDIA}
     \email{bariksomdatta97@gmail.com}
     
     \address[Ghosh]{Department of Mathematics\\  Jadavpur University\\ Kolkata 700032\\ West Bengal\\ INDIA}	
     \email{sghosh0019@gmail.com}
     
     \address[Paul]{Vice-Chancellor, Kalyani University \& Professor of Mathematics\\ Jadavpur University (on lien) \\ Kolkata \\ West Bengal\\ INDIA}
     \email{kalloldada@gmail.com}
     
     	\address[Sain]{Department of Mathematics, Indian Institute of Information Technology, Raichur, Karnataka, INDIA}
     		\email{saindebmalya@gmail.com}

	\begin{abstract}
		Given a nonzero positive operator $A$ on a Hilbert space $\mathbb{H}$, a semi-inner product is naturally induced on $\mathbb{H}$. In this work, we introduce  the notion of \emph{$A$-smoothness} for bounded linear operators on the resulting semi-Hilbert space and investigate its various properties. We provide a comprehensive characterization of the $A$-smoothness for $A$-bounded operators and further analyze the $A$-smoothness of $A$-compact operators in terms of their $A$-norm attainment sets. Utilizing these characterizations, we establish that G\^{a}teaux differentiability of the semi-norm $\|\cdot\|_A$ at an $A$-bounded operator is equivalent to its $A$-smoothness. Furthermore, we characterize the $A$-smoothness of $2\times 2$ block diagonal matrices.

	\end{abstract}
	
	\maketitle
	
	\section{introduction}
	
	
	The concept of \emph{G\^{a}teaux differentiability} of the norm function plays a central role in understanding the geometry of Banach spaces and the space of bounded linear operators. It corresponds to the notion of smoothness of a point, meaning that the point admits a unique supporting functional. Several works have contributed to this area of study (see \cite{CGPS-AFA,MPTS,PSG,Rao_LAA_17, Roy}). In this note, we aim to explore the G\^{a}teaux differentiability of a semi-norm within the setting of semi-Hilbertian operators. This investigation is carried out through a newly introduced concept known as \emph{A-smoothness}. Before presenting the main results, we begin by establishing some notations and terminology that will be used throughout this work.\\

	Let $(\mathbb{H}, \langle \cdot, \cdot \rangle)$ be a Hilbert space over the scalar field $\mathbb{R}$ or $\mathbb{C}.$  We denote $\Re(z)$ as the real part of $z \in \mathbb{C}.$ Let $\mathbb{B}(\mathbb H)$ be the collection of all bounded linear operators on $\mathbb{H}.$ An operator $A\in \mathbb{B}(\mathbb H)$ is said to be positive if $\langle Ax, x\rangle\geq0$ for all $x\in \mathbb H$. For any $T\in\mathbb{B}(\mathbb H),$  the range and the null space of $T$ are denoted by $R(T)$ and $N(T)$, respectively. Suppose that $A$ is a positive nonzero operator which generates a semi-inner product $\langle \cdot, \cdot\rangle_A$ on $\mathbb{H}$ , defined as $\langle x, y\rangle_A=\langle Ax, y\rangle.$ The semi-inner product induces a semi-norm $\|\cdot\|_A$, defined as $\|x\|_A=\sqrt{\langle x, x\rangle_A}=\|A^{1/2}x\|$ for all $x\in\mathbb H.$ Note that, $\|x\|_A=0$ if and only if $x\in N(A)$. The vector space $\mathbb H$, equipped with the semi-inner product $\langle \cdot, \cdot\rangle_A$, is referred to as a semi-Hilbert space. This study was initiated by Krein in \cite{Krein}. The semi-Hilbert space is complete if and only if $R(A)$ is closed in $\mathbb H$. The semi-inner product $\langle \cdot, \cdot\rangle_A$ induces an inner product
	on the quotient space $\mathbb H/N(A)$ defined as $[\bar x,\bar y ]=\langle x, y\rangle_A$ for all $\bar x,\bar y\in\mathbb H/N(A)$.
	The completion of $(\mathbb H/N(A), [\cdot, \cdot])$ is isometrically isomorphic to the Hilbert space $\left(R(A^{1/2}), \langle\cdot, \cdot\rangle_{\textbf{R}(A^{1/2})}\right)$, where the inner product is defined as $$\langle A^{1/2}x,A^{1/2}y\rangle_{\textbf{R}(A^{1/2})}=\langle P_{\overline{R(A)}}x, P_{\overline{R(A)}}y\rangle$$ for all $x, y\in\mathbb H.$ The Hilbert space $\left(R(A^{1/2}), \langle\cdot, \cdot\rangle_{\textbf{R}(A^{1/2})}\right)$ is denoted by $\textbf{R}(A^{1/2})$ and the norm associated with the inner product $\langle\cdot, \cdot\rangle_{\textbf{R}(A^{1/2})}$ is written as $\|\cdot\|_{\textbf{R}(A^{1/2})}$. 
	Note that, $R(A)$ is dense in $\textbf{R}(A^{1/2})$ (see \cite{ACG09}). We direct the reader to \cite{ACG09, ACG08, ACG_IEOT_08} for more insights into the space $\textbf{R}(A^{1/2})$. Let $	B_{\mathbb{H}(A)} = \{x \in \mathbb{H} : \|x\|_A \leq 1\}$ and $S_{\mathbb{H}(A)} = \{x \in \mathbb{H} : \|x\|_A = 1\}$ be the $A$-unit ball and the $A$-unit sphere of the semi-Hilbert space $\left(\mathbb{H}, \langle \cdot, \cdot\rangle_A\right),$ respectively. An operator $T \in \mathbb{B}(\mathbb{H})$ is said to be $A$-bounded if there exists a positive constant $c$ such that $ \|Tx\|_A \leq c \|x\|_A$  for all $x \in \mathbb{H}.$ The set of all such operators is denoted by $\mathbb B_{A^{1/2}}(\mathbb{H}).$	For any $T \in \mathbb B_{A^{1/2}}(\mathbb{H})$, the $A$-norm is given by
	\[
	\|T\|_A = \sup_{\|x\|_A = 1} \|Tx\|_A = \sup \{ |\langle Tx, y \rangle_A| : x, y \in \mathbb{H}, \ \|x\|_A = \|y\|_A = 1 \}.
	\]
	Let $T \in \mathbb B_{A^{1/2}}(\mathbb{H}).$ The $A$-norm attainment set of $T$ is defined as $
	M_A(T) = \{ x \in \mathbb{H} : \|x\|_A = 1,\ \|Tx\|_A = \|T\|_A \}.$ Whenever $A=I,$ we denote $M(T)$ as the usual norm attainment set of $T.$ For $T \in \mathbb{B}(\mathbb{H}),$ an operator $W \in \mathbb{B}(\mathbb{H})$ is called an $A$-adjoint of $T$ if $\langle Tx,y\rangle_A=\langle x,Wy\rangle_A$ for all $x, y\in \mathbb{H},$ equivalently, the equation $AX=T^*A$ has a solution. Not every $T\in\mathbb{B}(\mathbb{H})$  admits an $A$-adjoint. By Douglas theorem \cite{D66}, $T$ admits a unique $A$-adjoint if and only if $R(T^*A)\subset R(A).$ From now on, we denote by $\mathbb B_A(\mathbb{H})$ the set of all $T\in\mathbb{B}(\mathbb{H})$ admitting a unique $A$-adjoint: 
	\[\mathbb B_A(\mathbb{H})=\{T \in \mathbb{B}(\mathbb{H}): R(T^*A)\subset R(A)\}.\]
	If $T\in \mathbb B_A(\mathbb{H}),$ its $A$-adjoint is denoted by $T^{\sharp}$ and satisfies $R(T^\sharp)\subset\overline{R(A)}$. Note that $T^\sharp=A^\dagger T^*A$, where $A^\dagger$ is the Moore-Penrose inverse of $A,$ see \cite{P}. Moreover, $\mathbb B_A(\mathbb{H})\subset \mathbb B_{A^{1/2}}(\mathbb{H})\subset\mathbb{B}(\mathbb{H})$. We direct readers to \cite{ACG09, ACG08, ACG_IEOT_08,SSP_AFA_21,Z19} for further insights in this direction. $T \in \mathbb{B}(\mathbb{H})$ is said to be $A$-compact \cite{BKPS_23} if every bounded sequence  in $(\mathbb H, \|\cdot\|_A)$ has a convergent subsequence in $(\mathbb H, \|\cdot\|_A)$. 
	
	The notion of Birkhoff–James orthogonality has been explored in recent years for the study of geometric and analytic properties in the space of bounded linear operators (see \cite{MPS}). An operator $T\in\mathbb B(\mathbb H)$ is said to be Birkhoff-James orthogonal to $S\in\mathbb B(\mathbb H)$ if $\|T+\lambda S\|\geq \|T\|$ for all scalars $\lambda.$ Analogously, in the semi-Hilbertian setting, $T \in \mathbb{B}_{A^{1/2}}(\mathbb{H})$ is $A$-Birkhoff–James orthogonal to $S \in \mathbb{B}_{A^{1/2}}(\mathbb{H}),$ denoted by $T\perp_A^B S,$ (see \cite{Roy, SSP_AFA_21}) if $\|T+\lambda S\|_A\geq \|T\|_A$ for all scalars $\lambda.$ It is well known that the right additivity of Birkhoff-James orthogonality at a point of a normed linear space characterizes smoothness of the norm at that point. Motivated by this we introduce the notion of smoothness of operators in $\mathbb B_{A^{1/2}}(\mathbb{H})$ as follows:
	
	\begin{definition}
		Let $T \in \mathbb B_{A^{1/2}}(\mathbb{H})$ be such that $\|T\|_A\neq0.$ Then $T$ is said to be $A$-smooth if for any $S_1, S_2 \in B_{A^{1/2}}(\mathbb H)$ such that $T \perp_A^B S_1$ and $T \perp_A^B S_2$ imply $T \perp_A^B S_1+S_2.$
	\end{definition}
	
	In this work, we begin by characterizing the notion of $A$-smoothness for $A$-bounded operators in a semi-Hilbert space. We then focus on the case where the set $M_A(T)$ is non-empty, providing a refined characterization of $A$-smoothness under this condition. We establish that an $A$-compact operator is $A$-smooth if and only if the intersection $M_A(T) \cap \overline{R(A)}$ is a singleton up to scalar multiples. Furthermore, we demonstrate the equivalence between $A$-smoothness and the existence of the G\^{a}teaux derivative of the associated semi-norm function in semi-Hilbertian operator spaces. We also explore the relation between the $A$-smoothness of operators in $\mathbb{B}_{A^{1/2}}(\mathbb{H})$ and the smoothness of operators in $\mathbb{B}(\textbf{R}(A^{1/2}))$. Finally, we examine $A$-smoothness in the context of block diagonal matrices.

	\section{Main Results.}
	We first aim to characterize the newly introduced notion of $A$-smoothness. Before that we  note the following known results, which will be useful throughout this article.
	
	\begin{lemma}\cite{ACG09}\label{lemma1}
		For any $T\in \mathbb B_{A^{1/2}}(\mathbb{H}),$ there exists a unique $\widetilde{T} \in \mathbb B(\textbf{R}(A^{1/2}))$ such that $\widetilde{T}W_A=W_AT$, where $W_A: \mathbb{H}\to \textbf{R}(A^{1/2})$ satisfying $W_Ax=Ax$ for all $x \in \mathbb{H}.$
	\end{lemma}

	\begin{lemma}\label{lemma2}
		\begin{itemize}
			\item [(i)]  \cite{F_AFA_20} Let $T\in \mathbb B_{A^{1/2}}(\mathbb{H}).$ Then $\|\widetilde{T}\|_{\mathbb B(\textbf{R}(A^{1/2}))} = \|T\|_A$.
			\item [(ii)] \cite{MSS_LAMA_13} Let $T\in \mathbb B_A(\mathbb{H}).$ Then $\widetilde{T^{\sharp}}=(\widetilde{T})^*$ .
			\item [(iii)] \cite{BKPS_23} Let $T\in \mathbb B_{A^{1/2}}(\mathbb{H})$  and $R(A)$ be closed. Then $T$ is $A$-compact if and only if $\widetilde{T}$ is compact.
		\end{itemize}
		
	\end{lemma}
	
	In this connection we observe a relation between the $A$-norm attainment set of $T \in \mathbb B_{A^{1/2}}(\mathbb{H})$ and the norm attainment set of $\widetilde{T }\in  \mathbb B(\textbf{R}(A^{1/2})).$
	
	\begin{prop}\label{norm attainment}
		Let $T \in \mathbb B_{A^{1/2}}(\mathbb H)$ and $\widetilde{T} \in  \mathbb B(\textbf{R}(A^{1/2}))$ be as given in Lemma \ref{lemma1}. Then $x \in M_A(T)\cap \overline{R(A)}$ if and only if $Ax \in M(\widetilde{T}).$
	\end{prop}

	\begin{proof}
		Suppose that $x \in M_A(T)\cap \overline{R(A)}.$ This implies $\|Tx\|_A=\|T\|_A.$ By using Lemma \ref{lemma1} and Lemma \ref{lemma2} (i), we have
		\[\|\widetilde{T}\|_{\mathbb B(\textbf{R}(A^{1/2}))} = \|T\|_A= \|Tx\|_A = \|ATx\|_{\textbf{R}(A^{1/2})}= \|\widetilde{T}Ax\|_{\textbf{R}(A^{1/2})}.\]
		Also, note that 
		\begin{eqnarray*}
			\|Ax\|_{\textbf{R}(A^{1/2})}^2&=& \langle A^{1/2}(A^{1/2}x), A^{1/2}(A^{1/2}x)\rangle_{\textbf{R}(A^{1/2})}\\ &= &\langle P_{\overline{R(A)}}(A^{1/2}x), P_{\overline{R(A)}}(A^{1/2}x)\rangle \\ &=& \|A^{1/2}x\|^2=\|x\|^2_A =1.
		\end{eqnarray*}
		Therefore, $Ax\in M(\widetilde{T}).$ Using similar argument we can obtain the converse part. 
	\end{proof}
	
	Next, we mention the characterization of $A$-Birkhoff-James orthogonality, which plays a significant role in our whole scheme of things.
	\begin{theorem}\cite[Th. 2.2]{Z19}\label{zamani ch}
		Let $T, S \in \mathbb B_{A^{1/2}}(\mathbb{H}).$ Then the following conditions are equivalent:
		\begin{itemize}
			\item[(i)] $T\perp_A^B S$.
			\item[(ii)] There exists a sequence $\{x_n\}\in B_{\mathbb{H}(A)}$ such that $\lim\limits_{n\to \infty}\|Tx_n\|_A = \|T\|_A$ and $\lim\limits_{n\to\infty}\langle Tx_n, Sx_n\rangle=0. $
		\end{itemize}
	\end{theorem}
	
	In the following theorem we completely characterize $A$-smoothness of an $A$-bounded operator. 
	\begin{theorem}\label{Th_sing}
		Let $T \in \mathbb B_{A^{1/2}}(\mathbb{H})$ be such that $\|T\|_A\neq 0$. Then $T$ is $A$-smooth if and only if for any $S\in \mathbb B_{A^{1/2}}(\mathbb{H}),$ $W_A(T, S)$ is singleton, where 
		$$W_A(T, S)= \big\{ \sigma \in \mathbb{C}: \|x_n\|_A=1, \|Tx_n\|_A \to \|T\|_A \, \, \mbox{and}  \, \langle Tx_n, Sx_n\rangle_A \to \sigma\big\}.$$
	\end{theorem}
	
	\begin{proof}
		Since $A$-orthogonality is homogeneous, without loss of generality we may assume that $\|T\|_A= \|S\|_A=1.$ We first prove the necessary part. Suppose on the contrary that $W_A(T, S)$ is not singleton for some $S \in \mathbb B_{A^{1/2}}(\mathbb{H})$. Let $\alpha, \beta \in W_A(T, S)$ be such that $\alpha \neq \beta.$ Therefore, there exist $A$-norming sequnces $\{x_n\}$   and $\{y_n\}$ of $T$ such that $\langle Tx_n, Sx_n\rangle_A \to \alpha$ and $\langle Ty_n, Sy_n\rangle_A \to \beta.$ Suppose that $\alpha, \beta $ both are nonzero.  Now consider $S_1= S- \alpha T$ and $S_2=S-\beta T.$ Then it is easy to see that $\langle Tx_n, S_1x_n\rangle_A \to 0.$ This implies $T \perp_A^B S_1.$ Similarly, we can show that $T \perp_A^B S_2.$ As $T$ is $A$-smooth, it follows that $T \perp_A^B S_1 -S_2,$ i.e., $T \perp_A^B (\beta - \alpha)T.$ Since $\alpha \neq \beta,$  we get $\|T\|_A=0,$ which is a contradiction.
		Now suppose that at least one of $\alpha$ or $\beta$ zero. Let $\alpha=0.$ This implies $T \perp_A^B S.$ From previous argument we can show that $T\perp_A^B S-\beta T.$ Therefore, by $A$-smoothness of $T$ we get $T \perp_A^B S - \beta T - S.$ This implies $\|T\|_A=0,$ which is again a contradiction. \\
		To prove the sufficient part, let $T \perp_A^B S_1$ and $T \perp_A^B S_2$ for some $S_1, S_2 \in \mathbb B_{A^{1/2}}(\mathbb{H}).$ From Theorem \ref{zamani ch} we see that there exists $\{x_n\}$ with $\|x_n\|_A=1$ such that $\|Tx_n\|_A \to \|T\|_A$ and $\langle Tx_n, S_1x_n\rangle_A \to 0.$ Moreover, we have $W_A(T, S_1)$ is singleton. So we obtain that $W_A(T, S_1)=\{0\}$. Similarly, $W_A(T, S_2)=\{0\}.$ Therefore, we can take a common sequence $\{x_{n_k}\}\subset \{x_n\}$ with $\|Tx_{n_k}\|_A\to \|T\|_A$ such that $\langle Tx_{n_k}, S_1x_{n_k}\rangle_A \to 0$ and $\langle Tx_{n_k}, S_2x_{n_k}\rangle_A \to 0.$ This implies that $\langle Tx_{n_k}, (S_1+S_2)x_{n_k}\rangle_A \to 0$ and thus $T \perp_A^B S_1+S_2.$ This shows that $T$ is $A$-smooth.
	\end{proof}
	
	If we consider $A=I\in \mathbb{B}(\mathbb{H})$ then we have the following corollary regarding classical smoothness of operators defined on a Hilbert space, which is also proved in \cite{Roy}.
	
	\begin{cor}\label{SRoy}
		Let $T\in \mathbb{B}(\mathbb{H})$ be nonzero. Then $T$ is smooth if and only if $W(T, S)$ is singleton for every $S \in \mathbb{B}(\mathbb{H}),$ where $$W(T, S)= \big\{ \sigma \in \mathbb{C}: \|x_n\|=1, \|Tx_n\| \to \|T\| \, \, \mbox{and}  \, \langle Tx_n, Sx_n\rangle \to \sigma\big\}.$$
	\end{cor}

	Now we establish that $A$-compact operators possess a nonempty $A$-norm attainment set whenever $R(A)$ is closed.
	
	\begin{prop}\label{prop; nonempty}
		Let $T \in \mathbb B_{A^{1/2}}(\mathbb{H})$ be $A$-compact and let $R(A)$ be closed. Then $M_A(T)\neq \emptyset.$
	\end{prop}
	
	\begin{proof}
		Since $T$ is $A$-compact and $R(A)$ is closed, from Lemma \ref{lemma1} and Lemma \ref{lemma2} (iii), there exists a unique $\widetilde{T}\in \mathbb B(\textbf{R}(A^{1/2}))$ such that $\widetilde{T}$ is compact. Note that $\|T\|_A=\|\widetilde{T}\|_{\textbf{R}(A^{1/2})}.$ As $\textbf{R}(A^{1/2})$ is a Hilbert space and $\widetilde{T}$ is compact, one can easily observe that there exists $y_0 \in \textbf{R}(A^{1/2})$ with $\|y_0\|_{\textbf{R}(A^{1/2})}=1$ such that $\|\widetilde{T}y_0\|_{\textbf{R}(A^{1/2})}=\|\widetilde{T}\|_{\textbf{R}(A^{1/2})}.$ Since $R(A)$ is closed, there exists $x_0 \in \mathbb{H}$ such that $Ax_0=y_0.$ Clearly, $\|x_0\|_A=\|y_0\|_{\textbf{R}(A^{1/2})}=1.$ So we have $\|\widetilde{T}Ax_0\|_{\textbf{R}(A^{1/2})}=\|\widetilde{T}\|_{\textbf{R}(A^{1/2})}.$ In other words, $\|Tx_0\|_A=\|T\|_A.$ This shows that $M_A(T) \neq \emptyset.$
	\end{proof}
	
	We show by an example that the closedness of $R(A)$ cannot be omitted from the above proposition.
	
	\begin{example}
		Let $A: \ell_2 \to \ell_2$ be defined as $Ae_n=\frac{1}{n}e_n,$ where $\{e_n\}$ denotes the standard orthonormal basis of $\ell_2$. Clearly, $A$ is nonzero positive operator. Also, $R(A)$ is not closed. Consider $T\in \mathbb{B}(\ell_2)$ defined as $Te_n=\frac{1}{n^{3/2}}e_1.$  We first show that $T$ is $A$-bounded. Given any $x = (x_n) \in \ell_2,$ $\|x\|_A = \|A^{1/2}x\|=\left(\sum\limits_{n=1}^\infty \frac{|x_n|^2}{n}\right)^{1/2}.$ Now 
		\begin{eqnarray*}
			\|Tx\|_A^2 = \big\|\sum_{n=1}^\infty x_n Te_n\big\|_A^2
			&=& \big\|\sum_{n=1}^\infty \frac{x_n }{n^{3/2}}e_1\big\|_A^2\\
			&=&\big\|A^{1/2}\sum_{n=1}^\infty \frac{x_n }{n^{3/2}}e_1\big\|^2\\
			&=&\big\|\sum_{n=1}^\infty \frac{x_n }{n^{3/2}}e_1\big\|^2\\
			&=&\bigg|\sum_{n=1}^\infty \frac{x_n }{n^{3/2}}\bigg|^2\\
			&\leq& \bigg(\sum_{n=1}^\infty \frac{1}{n^2}\bigg)\bigg(\sum_{n=1}^\infty\frac{|x_n|^2}{n}\bigg)= \frac{\pi^2}{6}\|x\|_A^2.
		\end{eqnarray*}
		Thus, we get $\|Tx\|_A\leq \frac{\pi}{\sqrt{6}}\|x\|_A.$ This shows that $T$ is $A$-bounded, i.e., $T\in \mathbb B_{A^{1/2}}(\ell_2).$ Next, we claim that $\|T\|_A=\frac{\pi}{\sqrt{6}}.$ Note that 
		\begin{eqnarray}
			\|T\|_A= \sup_{\|x\|_A=1}\|Tx\|_A= \sup_{\|x\|_A=1}\left|\sum_{n=1}^\infty \frac{x_n }{n^{3/2}}\right|.
		\end{eqnarray}
		Let us 	consider that $y= (y_n) = (\frac{x_n}{\sqrt{n}})\in \ell_2$ and $z=(\frac{1}{n})\in \ell_2.$ Observe that $y=A^{1/2}x$ and therefore, $\|x\|_A=1\iff \|y\|=1.$ Then following the above relation we get 
		\begin{eqnarray} \|T\|_A= \sup_{\|y\|=1} |\langle y, z\rangle| = \|z\| = \frac{\pi}{\sqrt{6}}.\end{eqnarray} If possible let there exists $x=(x_n)\in \ell_2$ with $\|x\|_A=1$ such that $\|Tx\|_A=\frac{\pi}{\sqrt{6}}.$ Now, $\left|\sum\limits_{n=1}^\infty \frac{x_n }{n^{3/2}}\right|=|\langle y, z\rangle|=\|z\|=\|y\|\|z\|.$ Now by the equality condition of Cauchy-Schwarz inequality, we have $y=\lambda z,$ for some positive scalar $\lambda.$ Thus we obtain $x=(x_n)=(\frac{\lambda}{\sqrt{n}}).$ But note that $x\notin \ell_2.$ This shows that  $M_A(T)=\emptyset.$  On the other hand, since $R(T)$ is finite-dimensional, one can easily check that $T$ is $A$-compact. Thus, being $A$-compact, $T$ does not attain its $A$-norm whenever $R(A)$ is not closed.
	\end{example}
	
	In the following theorem, we completely characterize the $A$-smoothness of $A$-bounded operators provided that $M_A(T)$ is nonempty. This will lead to the characterization of the $A$-smoothness of $A$-compact operators.			
	
	\begin{theorem}\label{A-smooth}
		Let $T \in \mathbb B_{A^{1/2}}(\mathbb{H})$ be such that $\|T\|_A\neq 0$ and let $M_A(T)\neq \emptyset.$ Then the following are equivalent:
		\begin{itemize}
			\item[(i)] $T$ is $A$-smooth.
			\item[(ii)] $M_A(T)\cap \overline{R(A)}=\{\mu x_0: |\mu|=1\},$ for some $x_0 \in S_{\mathbb{H}(A)}$ and $\sup\{\|Ty\|_A: \langle x_0, y\rangle_A=0\, \, \mbox{with}\,\, \|y\|_A=1\}< \|T\|_A.$ 
		\end{itemize}
	\end{theorem}
	
	\begin{proof}
		(i) $\implies$ (ii):  Suppose on the contrary that $z_1, z_2 \in M_A(T)\cap \overline{R(A)}$ are such that $z_1 \neq \lambda z_2$ for any $\lambda\in \mathbb{C}.$ From \cite[Th. 2.4]{SSP_AFA_21} we note that $M_A(T)\cap \overline{R(A)}$  is $A$-unit sphere of some subspace of $\mathbb{H}.$  Thus, without loss of generality we may assume that $\langle z_1, z_2\rangle_A=0.$ Let $H_0=span\{z_1, z_2\}$ and $H_0^{\perp_A}=\{y\in \mathbb{H}: \langle x, y\rangle_A=0, x\in H_0\}.$ Observe that $H_0\cap H_0^{\perp_A}=\{0\}.$ So we can write $\mathbb{H}= H_0 \oplus H_0^{\perp_A}.$ Given any $z \in \mathbb{H}$ we write $z=\alpha z_1+ \beta z_2 +h$ for some $\alpha, \beta \in \mathbb{C}$ and $h\in H_0^{\perp_A}.$ Now we define $S_1, S_2\in \mathbb B(\mathbb{H})$ as $S_1z=\alpha_1 Tz_1$ and $S_2z=\beta Tz_2+Th.$ As $T$ is $A$-bounded, then clearly $S_1, S_2 \in \mathbb B_{A^{1/2}}(\mathbb{H})$. Note that $S_1z_2=0$ and $S_2z_1=0.$ This implies that $T \perp_A^B S_1$ and $T\perp_A^B S_2.$ Since $S_1+S_2=T,$ we arrive at a contradiction that $T$ is $A$-smooth. Thus, $M_A(T)\cap \overline{R(A)}=\{\mu x_0: |\mu|=1\}$ for some $x_0\in S_{\mathbb{H}(A)}.$ \\
		Now again suppose on the contrary that $\sup\{\|Ty\|_A: \langle x_0, y\rangle_A=0\, \, \mbox{with}\,\, \|y\|_A=1\}= \|T\|_A.$  For any $z\in \mathbb{H}$ we write $z= \alpha x_0+ h$ where $\alpha \in \mathbb{C}$ and $h\in H_0:=\{y\in \mathbb{H}: \langle x_0, y\rangle_A=0\}.$ Define $S_1, S_2$ as $S_1z=\alpha Tx_0$ and $S_2z=Th$ for all $z\in \mathbb{H}.$ From our assumption there exists $\{h_n\}\subset H_0$ with $\|h_n\|_A=1$ satisfying $\lim \|Th_n\|_A \to \|T\|_A$, we have $\langle Th_n, S_1h_n\rangle_A =0.$ This implies $T \perp_A^B S_1$. Whereas, $\langle Tx_0, S_2x_0\rangle_A=0$ implies $T \perp_A^B S_2.$ But observe that $T \not\perp_A^B S_1+S_2=T,$ which contradicts the fact that $T$ is $A$-smooth. This completes the proof of the theorem.\\

		(ii)$\implies$ (i):  Let $S_1, S_2 \in \mathbb B_{A^{1/2}}(\mathbb{H})$ be such that $T \perp_A^B S_i$ for each $i\in \{1, 2\}.$ Now we claim that $\langle Tx_0, S_ix_0\rangle_A=0$ for each $i\in \{1, 2\}.$ We only show it for $S_1$ as the case for $S_2$ follows similarly. From Theorem \ref{zamani ch} we note that there exists $\{z_n\} \subset S_{\mathbb{H}(A)}$ such that $\|Tz_n\|_A \to \|T\|_A$ and $\langle Tz_n, S_1 z_n\rangle_A\to 0.$ Suppose that $H_0=\{y\in \mathbb{H}: \langle x_0, y\rangle_A=0\}.$ Note that $x_0 \notin H_0$ otherwise $\|x_0\|_A=0.$ Therefore, we may write $z_n=\alpha_n x_0 + h_n,$ where $\alpha_n\in \mathbb{C}$ and $h_n \in H_0.$ As $\|z_n\|_A=1,$ it is easy to see that $|\alpha_n|\leq 1$ and $\|h_n\|_A^2=1-|\alpha_n|^2.$ Since $\{\alpha_n\} \subset \mathbb{C}$ is bounded, without loss of generality we can consider $\alpha_n \to \alpha_0$ with $|\alpha_0|\leq 1.$ Next, we show that $\lim\|h_n\|_A=0$. On the contrary suppose that $\lim\|h_n\|_A\neq0$. Without loss of generality, let $\|h_n\|_A\neq 0$ for all $n\in\mathbb N.$  From \cite[Th. 2.6]{MBPS_Arx}, whenever $x_0 \in M_A(T),$ we have $\langle x_0, y\rangle_A=0 \implies \langle Tx_0, Ty \rangle_A=0.$ Now
		\begin{eqnarray*}
			\lim \|Tz_n\|_A^2 &=& \lim \langle Tz_n, Tz_n\rangle_A\\
			&=& \lim\langle T(\alpha_n x_0+ h_n), T(\alpha_n x_0+ h_n)\rangle_A \\
			&=& \lim |\alpha_n|^2\|T\|_A^2 + \lim \|Th_n\|_A^2.\\
		\end{eqnarray*}	
		From the above equality we get $\lim\|Th_n\|_A^2 = (1- |\alpha_0|^2)\|T\|_A^2.$ Previously we already have $\|h_n\|_A^2 = 1- |\alpha_n|^2$ which implies $1-|\alpha_0|^2 = \lim \|h_n\|_A^2.$ So, we get $\lim \|Th_n\|_A^2 = \lim\|h_n\|_A^2\|T\|_A^2.$ This implies that $\lim \left\|T\left(\frac{h_n}{\|h_n\|_A}\right)\right\|_A = \|T\|_A,$ which contradicts our hypothesis. Thus, we have $\lim\|h_n\|_A=0$ and so $|\alpha_0|=1.$ Therefore, $\|Tz_n\|_A\to \|Tx_0\|_A=\|T\|_A$ will obtain $\langle Tz_n, S_1z_n\rangle_A \to 0 \implies \langle Tx_0, S_1x_0\rangle_A=0.$ This completes the proof of our claim. Now one can easily observe that $T \perp_A^B S_1$ and $T\perp_A^B S_2$ will imply $T\perp_A^B S_1+ S_2.$ This means $T$ is $A$-smooth.

	\end{proof}
	
	The characterization of Birkhoff-James orthogonality of operators in $\mathbb{B}(\mathbb
	H),$ popularly known as Bhatia-\v Semrl Theorem was proved in \cite{Bhatia} and  \cite{Paul} independently.  Zamani  generalizes this Bhatia-\v Semrl Theorem in the framework of semi-Hilbertian structure. An operator $T\in \mathbb B_{A^{1/2}}(\mathbb{H})$ is said to satisfy the Bhatia-\v Semrl Property (in short, B\v S Property) if for any $S \in \mathbb B_{A^{1/2}}(\mathbb{H}),$ there exists $x\in M_A(T)$ such that $T\perp_A^B S \iff \langle Tx, Sx\rangle_A=0.$ In this thread, we observe a connection between $A$-smooth operator and Bhatia-\v Semrl Property. 
	
	\begin{theorem}\label{BS}
		Let $T\in \mathbb B_{A^{1/2}}(\mathbb{H})$ be such that $\|T\|_A\neq 0$ and $M_A(T)\neq \emptyset.$ Then $T$ is $A$-smooth if and only if the following conditions hold:
		\begin{itemize}
			\item[(i)] $M_A(T)\cap \overline{R(A)}=\{\mu x_0: |\mu|=1\}.$
			\item[(ii)]  For any $S\in \mathbb B_{A^{1/2}}(\mathbb{H})$, $T\perp_A^B S\iff \langle Tx_0, Sx_0\rangle_A=0$, i.e., $T$ satisfies B\v S Property.
		\end{itemize}
	\end{theorem}

	\begin{proof}
		To prove the necessary part, we only prove (ii) as (i) follows directly from Theorem \ref{A-smooth}. To prove ``$\implies$'', suppose on the contrary there exists $S\in \mathbb B_{A^{1/2}}(\mathbb{H})$ such that $T\perp_A^B S$ but $\langle Tx_0, Sx_0\rangle_A \neq0.$  Consider $S'= T- \frac{\|T\|_A^2}{\langle Tx_0, Sx_0\rangle_A}S \in \mathbb B_{A^{1/2}}(\mathbb{H}).$ One can clearly observe that $T\perp_A^B S'.$ As $T$ is $A$-smooth, we obtain that $T\perp_A^B (S' + \frac{\|T\|_A^2}{\langle Tx_0, Sx_0\rangle_A}S),$ i.e., $T \perp_A^B T,$ which is a contradiction. The reverse implication  trivially follows. \\
		Next we show the sufficient part. Suppose that $S_1, S_2 \in \mathbb B_{A^{1/2}}(\mathbb{H})$ are such that $T\perp_A^B S_1$ and $T \perp_A^B S_2.$ Then we  have $\langle Tx_0, S_ix_0\rangle_A=0$ for each $i\in \{1, 2\}.$ From this we obtain $\langle Tx_0, (S_1+S_2)x_0\rangle_A=0.$ This proves that $T\perp_A^B S_1+S_2,$ i.e., $T$ is $A$-smooth.This completes the proof. 
	\end{proof}

	Next, we characterize the $A$-smoothness of $A$-compact operators by applying Theorem \ref{A-smooth}. 
	\begin{theorem}\label{cor; compact}
		Let $T $ be $A$-compact and let $ R(A)$ be closed. Then the following are equivalent:
		\begin{itemize}
			\item[(i)] $T$ is $A$-smooth.
			\item[(ii)] $M_A(T)\cap {R(A)}=\{\mu x_0: |\mu|=1\},$ for some $x_0 \in S_{\mathbb{H}(A)}.$
		\end{itemize}
	\end{theorem}

	\begin{proof}
		From Proposition \ref{prop; nonempty} we note that $M_A(T)\neq \emptyset.$ (i)$\implies$(ii) follows from Theorem \ref{A-smooth}. To  prove (ii)$\implies$(i), we only show that for any $A$-compact operator $T$, $$\sup\{\|Ty\|_A: \langle x_0, y\rangle_A=0\, \, \mbox{with}\,\, \|y\|_A=1\}< \|T\|_A.$$ Since $T$ is $A$-compact and $R(A)$ is closed, then from Lemma \ref{lemma1} and Lemma \ref{lemma2} (iii) we get a unique $\widetilde{T}\in \mathbb B(\textbf{R}(A^{1/2}))$ is also compact. As $M_A(T)\cap R(A) = \{\mu x_0: |\mu|=1, \|x_0\|_A=1\}$, using Proposition \ref{norm attainment} we get $M(\widetilde{T})=\{\mu Ax_0: |\mu|=1\}.$ Let $Ax_0=y_0.$ As $\widetilde{T}$ is compact, one can observe that (c.f. \cite[Remark 6.2.4]{MPS}) 
		\begin{eqnarray*}
			&&\big\{\|\widetilde{T}z\|_{\textbf{R}(A^{1/2})}: \|z\|_{\textbf{R}(A^{1/2})}=1, \, \langle y_0, z\rangle_{\textbf{R}(A^{1/2})}=0\big\}< \|\widetilde{T}\|_{\mathbb B(\textbf{R}(A^{1/2}))}.
		\end{eqnarray*}
		Since $R(A)$ is closed, we can take $z=Au$ for some $u\in\mathbb H.$ Then we have
		\begin{eqnarray*}
			&& \sup\big\{\|\widetilde{T}Au\|_{\textbf{R}(A^{1/2})}: \|Au\|_{\textbf{R}(A^{1/2})}=1, \, \langle Ax_0, Au\rangle_{\textbf{R}(A^{1/2})}=0\big\}<\|T\|_A\\
			&\implies& \sup\big\{\|ATu\|_{\textbf{R}(A^{1/2})}: \|Au\|_{\textbf{R}(A^{1/2})}=1, \, \langle Ax_0, Au\rangle_{\textbf{R}(A^{1/2})}=0\big\}<\|T\|_A\\
			&\implies& \sup\big\{\|Tu\|_A: \|u\|_A=1, \, \langle x_0, u\rangle_A=0\big\}<\|T\|_A.
		\end{eqnarray*}
		Now applying Theorem \ref{A-smooth} again, we prove the desired result.
	\end{proof}
	
	The following corollary is immediate.
	
	\begin{cor}\label{finite dim}
		Let $\mathbb{H}$ be a finite-dimensional Hilbert space and let $A\in \mathbb B(\mathbb{H})$ be positive. Then for $T\in \mathbb B_{A^{1/2}}(\mathbb{H})$ the following are equivalent:
		\begin{itemize}
			\item [(i)] $T$ is $A$-smooth.
			\item[(ii)] $M_A(T)\cap R(A)= \{\mu x_0: |\mu|=1, \|x_0\|_A=1\}.$
		\end{itemize}
	\end{cor}

	Let $T\in \mathbb B_{A^{1/2}}(\mathbb{H})$. The $A$-operator semi-norm $\|\cdot\|_A$  is a continuous convex function on the space $\mathbb B_{A^{1/2}}(\mathbb{H})$. Then for $S\in \mathbb B_{A^{1/2}}(\mathbb{H}),$ 
	\[\rho_{\pm}^A(T, S)=\lim_{t\to0^{\pm}}\frac{\|T+tS\|_A-\|T\|_A}{t}\] are said to be right-hand and left-hand G\^{a}teaux derivative of $\|\cdot\|_A$ at $T$ in the direction $S.$ $\|\cdot\|_A$ is G\^{a}teaux differentiable at $T$ if $\rho_{+}^A(T, S)=\rho_{-}^A(T, S)$ for all $S\in \mathbb B_{A^{1/2}}(\mathbb{H}).$ We refer the readers to \cite{F} for more detailed study on this topic.
	Using Theorem \ref{Th_sing} we show that the G\^{a}teaux derivative of the semi-norm in $\mathbb B_{A^{1/2}}(\mathbb{H})$ at a point is equivalent to the $A$-smoothness of that point. Before this we need the following lemma.
	
	\begin{lemma}\label{lemma; derivative range}
		Let $T \in \mathbb B_{A^{1/2}}(\mathbb{H})$ with $\|T\|_A\neq 0.$ Then for every $S \in \mathbb B_{A^{1/2}}(\mathbb{H})$ and $\lambda \in \mathbb{R},$ the following are equivalent:
		\begin{itemize}
			\item[(i)] $\lambda \in \Re (W_A(T, S)).$
			\item[(ii)] $\rho_{-}^A(T, S)\leq \lambda \leq \rho_{+}^A(T, S).$
		\end{itemize}
	\end{lemma}
	
	\begin{proof}
		Suppose that (i) holds true. Then there exists a sequence $\{x_n\}$ such that $\|x_n\|_A=1$ and $\lim \|Tx_n\|_A=\|T\|_A$ satisfying $\Re(\lim \langle Tx_n, Sx_n\rangle_A)=\lambda. $ For any $t>0,$ we have 
		\begin{eqnarray*}
			\lambda&=&\frac{1}{t}\Re(\lim \langle Tx_n, tSx_n\rangle_A )\\
			&=& \frac{1}{t}\Big(\Re(\lim \langle Tx_n, Tx_n+ tSx_n\rangle_A) - \Re(\lim\langle Tx_n, Tx_n\rangle_A)\Big)\\
			&\leq& \frac{\|T+tS\|_A - \|T\|_A}{t}.
		\end{eqnarray*}
		Therefore, $\lambda\leq \rho_+^A(T, S).$	On the other hand, we get
		\begin{eqnarray*}
			\frac{\|T-tS\|_A - \|T\|_A}{-t}= -\frac{\|T+ t(-S)\|_A-\|T\|_A}{t} &\leq& -\Re(\lim \langle Tx_n, -Sx_n\rangle_A)\\ &=& \Re(\lim\langle Tx_n, Sx_n\rangle_A). 
		\end{eqnarray*} This shows that $\lambda\geq \rho_-^A(T, S),$ which proves (ii).\\
		Now let (ii) hold true. For each $n,$ consider $x_n^\lambda \in \mathbb{H}$ with $\|x_n^\lambda\|_A=1$ such that $ \langle Tx_n^\lambda, rSx_n^\lambda\rangle_A=r(1-\frac{1}{n})\lambda$ for all $r\in \mathbb{R}.$ In other words, $x_n^\lambda \otimes Tx_n^\lambda$ is a linear functional on the subspace $\{rS: r\in \mathbb{R}\}$ dominated by the sublinear function $(1-\frac{1}{n})\rho_+^A(T, S).$ By Hahn-Banach extension, for any $S \in \mathbb B_{A^{1/2}}(\mathbb{H}),$ we observe that $(1-\frac{1}{n})\rho_-^A(T, S)\leq \Re(\langle Tx_n^\lambda, Sx_n^\lambda\rangle_A) \leq (1-\frac{1}{n})\rho_+^A (T, S).$ Since $\rho_+^A(T, T)=\rho_-^A(T, T)=\|T\|_A^2,$ it follows that $\lim\|Tx_n^\lambda\|_A^2=\|T\|_A^2.$ Together with the fact that $\lambda=\Re(\lim \langle Tx_n^\lambda, Sx_n^\lambda\rangle_A)$  we obtain  $\lambda \in \Re (W_A(T, S)).$ This completes the proof of the lemma.
	\end{proof}
	
	\begin{theorem}\label{Gateaux}
		Let $T\in \mathbb B_{A^{1/2}}(\mathbb{H})$ be such that $\|T\|_A\neq 0.$ Then $\|\cdot\|_A$ is G\^{a}teaux differentiable at $T $ if and only if $T$ is $A$-smooth.
	\end{theorem}
	
	\begin{proof}
		Suppose that $\|\cdot\|_A$ is G\^{a}teaux differentiable at $T$. Then for any $S \in \mathbb B_{A^{1/2}}(\mathbb{H})$ we have $\rho_-^A(T, S)=\rho_+^A(T, S).$ Then from Lemma \ref{lemma; derivative range} we have $\Re(W_A(T, S))$ is singleton for each  $S\in \mathbb B_{A^{1/2}}(\mathbb{H}).$ If possible let $\mu, \sigma \in W_A(T, S_0)$ for some $S_0\in \mathbb B_{A^{1/2}}(\mathbb{H}).$ Let $\mu=\lim \langle Tx_n, S_0x_n\rangle_A$ and $\sigma=\lim\langle Ty_n, S_0y_n\rangle_A$ for some $\{x_n\}, \{y_n\}\subset S_{\mathbb{H}(A)}$ satisfying $\lim \|Tx_n\|_A=\lim\|Ty_n\|_A=\|T\|_A.$ Now 
		\begin{eqnarray*}
			\mu&=& \lim  \langle Tx_n, S_0x_n\rangle_A\\
			&=& \Re  (\lim\langle Tx_n, S_0x_n\rangle_A)-i[\Re(i\lim \langle Tx_n, S_0x_n\rangle_A)]\\
			&=& \Re( \lim\langle Ty_n, S_0y_n\rangle_A) +i [\Re(\lim \langle Tx_n, iS_0x_n\rangle_A)]\\
			&=&  \Re(\sigma) +i [\Re(\lim  \langle Ty_n, iS_0y_n\rangle_A)]\\
			&=& \Re(\sigma) -i [\Re (i\lim  \langle Ty_n, S_0y_n\rangle_A)]\\
			&=& \Re(\sigma) -i \Re(i\sigma)=
			\sigma. 
		\end{eqnarray*}
		This proves that $W_A(T, S)$ is singleton for each $S \in \mathbb B_{A^{1/2}}(\mathbb{H})$ and therefore, from Theorem \ref{Th_sing} we get $T$ is $A$-smooth.\\
		Conversely, if $T$ is $A$-smooth then from Theorem \ref{Th_sing} we have $W_A(T, S)$ is singleton for every $S \in \mathbb B_{A^{1/2}}(\mathbb{H}).$ From \cite[Lemma 2.1]{Z19} we get $\Re(W_A(T, S))$ is compact. Thus applying Lemma \ref{lemma; derivative range} we get $\rho_-^A(T, S)=\rho_+^A(T, S)$ for every $S \in  \mathbb B_{A^{1/2}}(\mathbb{H}).$ This proves that $\|\cdot\|_A$ is G\^{a}teaux differentiable at $T.$ Hence the theorem.
	\end{proof}

	Theorem \ref{Gateaux} reveals that the problem of whether  the $A$-operator semi-norm at a point $T\in \mathbb B_{A^{1/2}}(\mathbb{H})$ is  G\^{a}teaux differentiable or not can be tackled by the concept of $A$-smoothness of $T.$ In this connection we note the following examples:
	\begin{example}
		(i)  Consider that $\mathbb{H}=\mathbb{R}^2$ endowed with its usual inner product norm. Let $T\in \mathbb{B}(\mathbb{H})$ be the identity matrix. As $M(T)=S_{\mathbb{H}},$ $T$ is not smooth. So the operator norm on $\mathbb{B}(\mathbb{H})$ is not G\^{a}teaux differentiable at $T.$ Let 
		$A= \begin{pmatrix}
			1 & 0\\
			0 & 0
		\end{pmatrix}$. Clearly, $A$ is positive. Also, note that $R(A)=\{(x, 0)^t: x\in \mathbb{R}\}.$ For any $x\in \mathbb{H}$ satisfying $\|x\|_A=1,$ it follows that $x=\pm(1, 0)^t.$ So $M_A(T)\cap R(A)=\{\pm(1, 0)^t\}.$ From Corollary \ref{finite dim}, we get $T$ is $A$-smooth. Therefore, from Theorem \ref{Gateaux}, we observe that the $A$-operator semi-norm is G\^{a}teaux differentiable at $T.$	In fact, it is easy to observe that whenever $rank(A)=1$, every $T\in \mathbb B_{A^{1/2}}(\mathbb{H})$ with $\|T\|_A\neq 0$ is $A$-smooth. \\
		
		\noindent (ii)  Let $\mathbb{H}=\mathbb{R}^3.$ Consider the diagonal matrix $T= diag \{2, 1, 1\}$ and let $A=diag\{0, 1, 1\}.$  Clearly, $A$ is positive and $R(A)=\{(0, y, z)^t: y, z\in \mathbb{R}\}.$  Also note that $M(T)=\{\pm (1, 0, 0)^t\}.$ Therefore, $T$ is smooth. On the other hand,  $\|(0,1,0)^t\|_A=\|(0, 0, 1)^t\|_A=1$ and $\|T(0, 1, 0)^t\|_A= \|A^{1/2}(0, 1, 0)^t\|=1=\|T\|_A.$ So, $\pm(0, 1, 0)^t \in M_A(T).$ Similarly, we get $(0, 0, 1)^t\in M_A(T).$ It follows from Corollary \ref{finite dim} that $T$ is not $A$-smooth. Therefore, from Theorem \ref{Gateaux} we get that the function $\|\cdot\|_A$ on $\mathbb B_{A^{1/2}}(\mathbb{H})$ is not G\^{a}teaux differentiable at $T.$
	\end{example} 
	
	In the following result we compare the $A$-smoothness in $\mathbb B_{A^{1/2}}(\mathbb{H})$ with that in $\mathbb B(\textbf{R}(A^{1/2}).$ 	
	\begin{theorem}\label{Th_tild}
		Let $T \in \mathbb B_{A^{1/2}}(\mathbb{H})$. Then $T$ is $A$-smooth if and only if $\widetilde{T}$ is smooth.
	\end{theorem}
	\begin{proof}
		We first prove the necessary part.  Since $T \in \mathbb B_{A^{1/2}}(\mathbb{H}),$ it follows from Lemma \ref{lemma1} that there exists a unique $\widetilde{T}\in \mathbb B(\textbf{R}(A^{1/2})) $ such that $\widetilde{T}A=AT$. If possible suppose that $\widetilde{T}$ is not smooth. Then from Corollary \ref{SRoy} there exist $\widetilde{S}\in \mathbb B(\textbf{R}(A^{1/2}))$ and two sequences $\{x_n\}, \{y_n\}$ with $\|x_n\|_{\textbf{R}(A^{1/2})}=\|y_n\|_{\textbf{R}(A^{1/2})}=1$ and $\|\widetilde{T}x_n\|_{\textbf{R}(A^{1/2})}\to\|\widetilde{T}\|_{\textbf{R}(A^{1/2})}, \|\widetilde{T}y_n\|_{\textbf{R}(A^{1/2})}\to\|\widetilde{T}\|_{\textbf{R}(A^{1/2})}$ such that $$ \langle \widetilde{T}x_n, \widetilde{S}x_n\rangle_{\textbf{R}(A^{1/2})}\to \lambda_0, \langle \widetilde{T}y_n, \widetilde{S}y_n\rangle_{\textbf{R}(A^{1/2})}\to \sigma_0,$$ where $\lambda_0\neq \sigma_0.$
		As for each $n,$ $x_n, y_n \in R(A^{1/2}),$ we write $A^{1/2}u_n=x_n$ and $A^{1/2}v_n=y_n$ for some $u_n, v_n \in \mathbb{H}.$ Also, note that $R(A)$ is dense in $\textbf{R}(A^{1/2}).$ Therefore, for each $n,$ there exist $\{u_{n, k}\}, \{v_{n, k}\} \subset \mathbb H$ such that $\lim\limits_{k\to\infty}Au_{n, k} = A^{1/2}u_n$ and $\lim\limits_{k\to\infty} Av_{n, k} = A^{1/2}v_n$ in $\|\cdot\|_{\textbf{R}(A^{1/2})}.$ This implies $\lim\limits_{k\to\infty} \|Au_{n, k}\|_{\textbf{R}(A^{1/2})}=\lim\limits_{k\to \infty}\|Av_{n, k}\|_{\textbf{R}(A^{1/2})}=1$ for each $n.$ Also, we have 
		\begin{eqnarray}
			\lim_{n\to\infty, k\to \infty}\|\widetilde{T}Au_{n, k}\|_{\textbf{R}(A^{1/2})}= \|\widetilde{T}\|_{\textbf{R}(A^{1/2})}=\|T\|_A,\\
			\lim_{n\to\infty, k\to \infty} \langle \widetilde{T}Au_{n, k}, \widetilde{S}Au_{n, k}\rangle_{\textbf{R}(A^{1/2})} =\lambda_0.\\
			\lim_{n\to\infty, k\to \infty}\|\widetilde{T}Av_{n, k}\|_{\textbf{R}(A^{1/2})}= \|\widetilde{T}\|_{\textbf{R}(A^{1/2})}=\|T\|_A,\\
			\lim_{n\to\infty, k\to \infty} \langle \widetilde{T}Av_{n, k}, \widetilde{S}Av_{n, k}\rangle =\sigma_0.
		\end{eqnarray}
		Now we consider $A$-normalized subsequences of $\{u_{n, k}\}$ and $\{v_{n, k}\}$, respectively as the following:
		\[z_r=\frac{u_{n_r, k_r}}{\|Au_{n_r, k_r}\|_{\textbf{R}(A^{1/2})}},
		w_s=\frac{v_{n_s, k_s}}{\|Av_{n_s, k_s}\|_{\textbf{R}(A^{1/2})}}.\] Note that $\|z_r\|_A=\|w_s\|_A=1$ for all $r$ and $s.$ From the above equations (2.3) and (2.4) together with the fact that $\lim\limits_{k\to\infty} \|Au_{n, k}\|_{\textbf{R}(A^{1/2})}=1$ we have 
		\begin{eqnarray*}
			&&\|\widetilde{T}Az_r\|_{\textbf{R}(A^{1/2})}\to\|\widetilde{T}\|_{\textbf{R}(A^{1/2})}\\
			&\implies&\|ATz_r\|_{\textbf{R}(A^{1/2})}\to\|\widetilde{T}\|_{\textbf{R}(A^{1/2})}\\
			&\implies& \|Tz_r\|_A\to\|T\|_A
		\end{eqnarray*}
		and 
		\begin{eqnarray*}
			&&\langle \widetilde{T}Az_r, \widetilde{S}Az_r\rangle_{\textbf{R}(A^{1/2})}\to \lambda_0\\
			&\implies&\langle ATz_r, ASz_r\rangle_{\textbf{R}(A^{1/2})}\to \lambda_0\\
			&\implies&\langle Tz_r, Sz_r\rangle_{A}\to \lambda_0.
		\end{eqnarray*}
		
		Similarly, from equations  (2.5) and (2.6), we can show that  $\|Tw_s\|_A\to\|T\|_A$ and $\langle Tw_s, Sw_s\rangle_{A}\to \sigma_0.$ Hence $\lambda_0, \sigma_0\in W_A(T, S).$ From Theorem \ref{Th_sing} we get that $T$ is not $A$-smooth, which is a contradiction. Therefore, we obtain that $\widetilde{T}$ is smooth.\\
		To prove the sufficient part, suppose on the contrary that $T$ is not $A$-smooth. Following Theorem \ref{Th_sing}, for some $S\in \mathbb B_{A^{1/2}}(\mathbb{H}),$ there exist $A$-norming sequences $\{x_n\}, \{y_n\}$ of $T$ such that $\langle Tx_n, Sx_n\rangle_A \to \lambda$ and $\langle Ty_n, Sy_n\rangle_A \to \mu$ where $\lambda \neq \mu.$ As $\|Ax_n\|_{\textbf{R}(A^{1/2})}=\|x_n\|_A,$ it is easy to see $\|\widetilde{T}Ax_n\|_{\textbf{R}(A^{1/2})}\to \|\widetilde{T}\|_{\textbf{R}(A^{1/2})}.$ Moreover,  $\langle \widetilde{T}Ax_n, \widetilde{S}Ax_n\rangle_{\textbf{R}(A^{1/2})} \to \lambda.$ Similarly, we obtain that $\|\widetilde{T}Ay_n\|_{\textbf{R}(A^{1/2})}=\|\widetilde{T}\|_{\textbf{R}(A^{1/2})}$ and $\langle \widetilde{T}Ay_n, \widetilde{S}Ay_n\rangle_{\textbf{R}(A^{1/2})}\to \mu.$ This contradicts that $\widetilde{T}$ is smooth (see Corollary \ref{SRoy}). Hence the theorem.
	\end{proof}  
	
	Applying the above theorem we note the following corollary.
	
	\begin{cor}
		Let $T \in \mathbb B_A(\mathbb{H})$ be $A$-compact and let $R(A)$ be closed. Then $T$ is $A$-smooth if and only if $T^\sharp$ is $A$-smooth.
	\end{cor}
	
	\begin{proof}
		Let $T$ be $A$-smooth. It follows from Theorem \ref{Th_tild} that $\widetilde{T}$ is smooth. Also, from Lemma \ref{lemma2} (iii) we get $\widetilde{T}$ is compact. Thus $(\widetilde{T})^*$ is smooth (see \cite[Th. 1]{Rao_LAA_17}). Note that, $(\widetilde{T})^*=\widetilde{T^\sharp}.$ This implies that $\widetilde{T^\sharp}$ is smooth. Therefore, again following Theorem \ref{Th_tild} we have $T^\sharp$ is $A$-smooth. The converse part follows using the same argument.
	\end{proof}
	
	We conclude this article by characterizing the $A$-smoothness of block diagonal matrices using Theorem \ref{Th_sing}. To this end, we first observe the following lemma.

	Here we consider $\mathbb A=\begin{pmatrix}
		A&0\\
		0&A
	\end{pmatrix}\in\mathbb B(\mathbb H\oplus\mathbb H),$ where $A\in \mathbb B(\mathbb H)$ is a positive operator. Clearly, $\mathbb A$ is a positive operator on $\mathbb H\oplus\mathbb H$ and it generates the semi-inner product $\langle x, y\rangle_{\mathbb A}=\langle x_1, y_1\rangle_{A}+\langle x_2, y_2\rangle_{A}$ for all $x=\begin{pmatrix}
		x_1\\
		x_2
	\end{pmatrix}\in\mathbb H\oplus\mathbb H$ and $y=\begin{pmatrix}
		y_1\\
		y_2
	\end{pmatrix}\in\mathbb H\oplus\mathbb H.$

	\begin{lemma}\label{lm_12}
		Let $T=\begin{pmatrix}
			M&0\\
			0&N
		\end{pmatrix}\in B_{\mathbb A^{1/2}}(\mathbb H\oplus\mathbb H)$ be such that $\|M\|_A>\|N\|_A$ and let $\begin{pmatrix}
			x_n\\
			y_n
		\end{pmatrix}$ be an $\mathbb A$-norming sequence of $T$. Then 
		
		\begin{itemize}
			\item [(i)]$\lim_{n\to\infty} \|x_n\|_A=1.$
			\item [(ii)]$\lim_{n\to\infty} \|y_n\|_A=0.$
			\item [(iii)] $\left\{\frac{x_n}{\|x_n\|_A}\right\}$ is an $A$-norming sequence of $M$.
		\end{itemize}
	\end{lemma}

	\begin{proof}
		Since $\begin{pmatrix}
			x_n\\
			y_n
		\end{pmatrix}$ is an $\mathbb A$-norming sequence of $T$ and $\|M\|_A>\|N\|_A,$ we have \[\lim_{n\to\infty}\left(\|Mx_n\|_A^2+\|Ny_n\|_A^2\right)^{1/2}=\|M\|_A\]
		and $\|x_n\|_A^2+\|y_n\|_A^2=1.$ Now, 
		\begin{eqnarray*}
			\|Mx_n\|_A^2+\|Ny_n\|_A^2
			&\leq& \|M\|_A^2\|x_n\|_A^2+\|N\|_A^2\|y_n\|_A^2\\
			&=& \|M\|_A^2(\|x_n\|_A^2+\|y_n\|_A^2)-(\|M\|_A^2-\|N\|_A^2)\|y_n\|_A^2\\
			&=& \|M\|_A^2-(\|M\|_A^2-\|N\|_A^2)\|y_n\|_A^2.
		\end{eqnarray*}
		Thus, 
		\begin{eqnarray*}
			&&\lim_{n\to\infty}\left(\|Mx_n\|_A^2+\|Ny_n\|_A^2\right)\leq\|M\|_A^2-(\|M\|_A^2-\|N\|_A^2)\limsup_{n\to\infty} \|y_n\|_A^2\\
			&\implies&\|M\|_A^2\leq\|M\|_A^2-(\|M\|_A^2-\|N\|_A^2)\limsup_{n\to\infty} \|y_n\|_A^2\\
			&\implies&-(\|M\|_A^2-\|N\|_A^2)\limsup_{n\to\infty} \|y_n\|_A^2\geq0\\
			&\implies&\limsup_{n\to\infty} \|y_n\|_A\leq0\\
			&\implies&\limsup_{n\to\infty} \|y_n\|_A=0\\
			&\implies&\lim_{n\to\infty} \|y_n\|_A=0.
		\end{eqnarray*}
		So, $\lim_{n\to\infty} \|x_n\|_A=1.$
		Now, $\lim_{n\to\infty}\|Mx_n\|_A^2=\|M\|_A^2-\lim_{n\to\infty}\|Ny_n\|_A^2=\|M\|_A^2.$ If we consider the sequence $\left\{\frac{x_n}{\|x_n\|_A}\right\}$ then it becomes the $A$-norming sequence of $M$.
	\end{proof}

	\begin{remark}\label{Rem_norm}
		It is easy to observe that if $\{x_n\}$ is an $A$-norming sequence of $M$ then $\begin{pmatrix}
			x_n\\
			0
		\end{pmatrix}$ is an $\mathbb A$-norming sequence of $T,$ where $\|M\|_A\geq\|N\|_A$.
	\end{remark}
	
	\begin{theorem}
		Let $T=\begin{pmatrix}
			M&0\\
			0&N
		\end{pmatrix}\in \mathbb B_{\mathbb A}(\mathbb H\oplus\mathbb H).$ 
		
		\begin{itemize}
			\item [(i)] If $\|M\|_A>\|N\|_A$ then $T$ is $\mathbb A$-smooth if and only if $M$ is $A$-smooth. 
			\item [(ii)] If $\|M\|_A<\|N\|_A$ then $T$ is $\mathbb A$-smooth if and only if $N$ is $A$-smooth. 
			\item [(iii)] If $\|M\|_A=\|N\|_A$ then $T$ is not $\mathbb A$-smooth.
		\end{itemize}

	\end{theorem}
	
	\begin{proof}
		
		(i) Let $\|M\|_A>\|N\|_A.$   If possible suppose that $M$ is not $A$-smooth. Then there exist $R\in \mathbb B_{A}(\mathbb H)$ and two sequences $\{x_n\}, \{y_n\}$ with $\|x_n\|_A=\|y_n\|_A=1$ and $\|Mx_n\|_A\to\|M\|_A, \|My_n\|_A\to\|M\|_A$ such that 
		\begin{eqnarray}\label{eq_1}
			\langle Mx_n, Rx_n\rangle_A\to \lambda \,\,\text{and}\,\,
			\langle My_n, Ry_n\rangle_A\to \mu, 
		\end{eqnarray}
		where $\lambda\neq\mu$. It follows from Remark \ref{Rem_norm} that $\begin{pmatrix}
			x_n\\
			0
		\end{pmatrix}$ and $\begin{pmatrix}
			y_n\\
			0
		\end{pmatrix}$ are two $\mathbb A$-norming sequences of $T.$  Now consider $S=\begin{pmatrix}
			R &Q\\
			U&V
		\end{pmatrix}\in \mathbb B_{\mathbb A}(\mathbb H\oplus\mathbb H).$ One can clearly observe that $$ \lim \Bigg\langle T \begin{pmatrix}
			x_n\\
			0
		\end{pmatrix}, S \begin{pmatrix}
			x_n\\
			0
		\end{pmatrix}\Bigg\rangle_{\mathbb{A}} = \lim \langle Mx_n, Rx_n\rangle_A = \lambda.$$
		Similarly we get $$ \lim \Bigg\langle T \begin{pmatrix}
			y_n\\
			0
		\end{pmatrix}, S \begin{pmatrix}
			y_n\\
			0
		\end{pmatrix}\Bigg\rangle_{\mathbb{A}} = \lim \langle My_n, Ry_n\rangle_A = \mu.$$ As $\lambda \neq \mu,$ it follows from Theorem \ref{Th_sing} that $T$ is not $\mathbb{A}$-smooth, which is a contradiction.\\
		
		Conversely, let $M$ be $A$-smooth. Then $W_A(M, P)$ is singleton for all $P\in \mathbb B_{A}(\mathbb H).$  If possible suppose that $T$ is not $\mathbb A$-smooth. Then there exists $S=\begin{pmatrix}
			P &Q\\
			U&V
		\end{pmatrix}\in \mathbb B_{\mathbb A}(\mathbb H\oplus\mathbb H)$ and two $\mathbb A$-norming sequences  
		$\begin{pmatrix}
			x_n\\
			y_n
		\end{pmatrix},$ $\begin{pmatrix}
			u_n\\
			v_n
		\end{pmatrix}$ of $T$ such that 
		\begin{eqnarray*}\label{eq_2}
			&& \left\langle T
			\begin{pmatrix}
				x_n\\
				y_n
			\end{pmatrix}, S\begin{pmatrix}
				x_n\\
				y_n
			\end{pmatrix}\right\rangle_\mathbb A\to \lambda \nonumber\\
			&\implies& \langle Mx_n, Px_n\rangle_A+\langle Mx_n, Qy_n\rangle_A+\langle Ny_n, Ux_n\rangle_A+\langle Ny_n, Vy_n\rangle_A \to\lambda\nonumber\\
			&\implies& \langle Mx_n, Px_n\rangle_A+\langle Q^\sharp Mx_n, y_n\rangle_A+\langle y_n, N^\sharp Ux_n\rangle_A+\langle y_n, N^\sharp Vy_n\rangle_A \to\lambda.
		\end{eqnarray*}
		From Lemma \ref{lm_12}, we have $\lim\limits_{n\to\infty}\|y_n\|_A=0$ and so 
		\begin{eqnarray}
			\langle Mx_n, Px_n\rangle_A\to \lambda.
		\end{eqnarray}
		Similarly, from $\left\langle T
		\begin{pmatrix}
			u_n\\
			v_n
		\end{pmatrix}, S\begin{pmatrix}
			u_n\\
			v_n
		\end{pmatrix}\right\rangle_\mathbb A\to \mu,$ we have
		\begin{eqnarray}\label{eq_33}
			\langle Mu_n, Pu_n\rangle_A\to \mu,
		\end{eqnarray}
		where $\lambda\neq\mu$. It follows from Lemma \ref{lm_12} that $\left\{\frac{x_n}{\|x_n\|_A}\right\}$ and $\left\{\frac{u_n}{\|u_n\|_A}\right\}$ are two $A$-norming sequences of $M.$ Since $M$ is $A$-smooth, we have $\langle Mx_n,Px_n\rangle_A\to\mu_0$ and $\langle Mu_n,Pu_n\rangle_A\to\mu_0$. From \eqref{eq_2} and \eqref{eq_33}, we have $\lambda=\mu=\mu_0$. This contradicts the fact that $\lambda\neq\mu$. Therefore, $T$ is $\mathbb A$-smooth.
		
		(ii) The proof of (ii) follows by using the same argument as in (i).
		
		(iii) Let $\|M\|_A=\|N\|_A.$ Suppose $\{x_n\}$ and $\{y_n\}$ are two $A$-norming sequences of $M$ and $N,$ respectively. Let $P\in \mathbb B_{A}(\mathbb H)$ be such that 
		$\langle Mx_n, Px_n\rangle_A\to \lambda (\neq 0).$ Since $\{x_n\}$ and $\{y_n\}$ are two $A$-norming sequences of $M$ and $N,$ then  $\begin{pmatrix}
			x_n\\
			0
		\end{pmatrix}$ and $\begin{pmatrix}
			0\\
			y_n
		\end{pmatrix}$ are two $\mathbb A$-norming sequences of $T.$ Suppose  $S=\begin{pmatrix}
			P &0\\
			0&0
		\end{pmatrix}\in \mathbb B_{\mathbb A}(\mathbb H\oplus\mathbb H).$ Then $\langle Mx_n, Px_n\rangle_A\to \lambda (\neq 0)$ and  $\langle Ny_n, Py_n\rangle_A\to 0.$ Therefore, from Theorem \ref{Th_sing}, we get $T$ is not $\mathbb A$-smooth. 	
	\end{proof}

	\subsection*{Acknowledgements}
	 Somdatta Barik and Souvik Ghosh would like to thank UGC, Govt. of India and CSIR, Govt. of India, respectively, for the support	in the form of Senior Research Fellowship under the mentorship of Professor Kallol Paul.
	 The research of Professor Kallol Paul is supported by CRG Project  bearing File no. CRG/2023/00716 of DST-SERB, Govt. of India.
	 
	
	\subsection*{Declarations}
	
	\begin{itemize}
		
		\item Conflict of interest
		
		The authors have no relevant financial or non-financial interests to disclose.
		
		\item Data availability 
		
		The manuscript has no associated data.
		
		\item Author contribution
		
		All authors contributed to the study. All authors read and approved the final version of the manuscript.
		
	\end{itemize}


\begin{thebibliography}{99}
		
		
		\bibitem{ACG09} M. L. Arias, G. Corach and M. C. Gonzalez, \textit{Lifting properties in operator ranges}, Acta Sci. Math.
		(Szeged), 75 (2009), 635--653.
		
		\bibitem{ACG08} M. L. Arias, G. Corach and M. C. Gonzalez, \textit{Partial isometries in semi-Hilbertian spaces}, Linear
		Algebra Appl., 428 (2008), 1460--1475.
		
		\bibitem{ACG_IEOT_08} M. L. Arias, G. Corach and M. C. Gonzalez, \textit{Metric properties of projections in semi-Hilbertian spaces}, Integr. Equ. Oper. Theory, 62 (2008), 11--28.
		
		
		\bibitem{Bhatia} R. Bhatia and P. Šemrl, \textit{Orthogonality of matrices and some distance problems}, Linear Algebra Appl., 287 (1999), 77–85.
		
		\bibitem{BKPS_23} P. Bhunia, F. Kittaneh, K. Paul and A. Sen, \textit{Anderson's theorem and $A$-spectral radius bounds for semi-Hilbertian space operators}, Linear Algebra Appl., 657 (2023), 147–162.
		
		\bibitem{CGPS-AFA} J. Chmieli\'nski, S. Ghosh,  K. Paul and D. Sain, \textit{Smoothness and approximate smoothness in normed linear spaces and operator spaces}, Ann. Funct. Anal., 16 (2025), no. 23, 1-24. 
		
		\bibitem{D66} R. G. Douglas, \textit{On majorization, factorization and range    inclusion of operators in Hilbert space}, Proc. Amer. Math. Soc., 17 (1966), 413--416.
		
		\bibitem{F} M. Fabian, \textit{G\^ateaux differentiability of convex functions and topology: weak Asplund spaces}, Wiley-Interscience, New York, 1997. 
		
		\bibitem{F_AFA_20} K. Feki, \textit{Spectral radius of semi-Hilbertian space operators and its applications}, Ann. Funct. Anal., 11(4) (2020), 929–946.
		
		\bibitem{Krein} M. G. Krein, \textit{Compact linear operators on functional spaces with two norms}, Integral Equ. Oper. Theory, 30 (1998), 140-162.
		
		\bibitem{MSS_LAMA_13} W. Majdak, N. A. Secelean and L. Suciu,\textit{ Ergodic properties of operators in some semi-Hilbertian spaces}, Linear Multilinear Algebra, \textbf{61} (2) (2013), 139–159.
		
		\bibitem{MPS} A. Mal, K. Paul and D. Sain, \textit{Birkhoff–James Orthogonality and Geometry  of Operator Spaces}, Infosys Science Foundation Series, Springer, Singapore, 2024.
		
		
		\bibitem{MPTS}  A. Mal, K. Paul, T.S.S.R.K. Rao and D. Sain, \textit{Approximate Birkhoff-James orthogonality and smoothness in the space of bounded linear operators}, Monatsh. Math., 190 (2019), 549--558.
		
		
		\bibitem{MBPS_Arx} J. Manna, S. Barik, K. Paul and  D. Sain, \textit{On $A$-orthogonality preservation and Blanco-Koldobsky-Turn\v{s}ek theorem in semi-Hilbert spaces}. 
		https://doi.org/10.48550/arXiv.2507.18871
		
		
		
		
		\bibitem{PSG} K. Paul, D. Sain and P. Ghosh,
		\textit{Birkhoff-James orthogonality and smoothness of bounded linear operators}, Linear Algebra Appl., 506 (2016), 551–563.
		
		\bibitem{Paul} K. Paul, \textit{Translatable radii of an operator in the direction of another operator}, Scientiae Math., 2 (1999), 119-122.
		
		\bibitem{P} R. Penrose, \textit{A generalized inverse for matrices}, Proc. Cambridge Philos. Soc., 51 (1955), 406-413.
		
		\bibitem{Rao_LAA_17}  T.S.S.R.K. Rao, \textit{Smooth points in spaces of operators},
		Linear Algebra Appl., 517 (2017), 129–133.
		
		
		\bibitem{Roy} S. Roy, \textit{The weak differentiability of norm and a generalized Bhatia-\v{S}emrl theorem}, Linear Algebra Appl., 685 (2024), 46--65.
		
		\bibitem{SSP_AFA_21} J. Sen, D. Sain and K. Paul, \textit{Orthogonality and norm attainment of operators in semi-Hilbertian spaces}, Ann. Funct. Anal., 12 (2021), 12 pp.
		
		
		
		\bibitem{Z19} A. Zamani, \textit{Birkhof–James orthogonality of operators in semi-Hilbertian spaces and its applications}, Ann. Funct. Anal., 10 (2019), 433--445. 
		
		
		
		
	\end{thebibliography}
\end{document}